\documentclass[11pt]{amsart}

\usepackage{amssymb}
\usepackage{amsmath,amssymb,epsf,wrapfig,multicol,epic,eepic,trig,mathrsfs,dsfont,color,graphicx}
 \usepackage[abs]{overpic}
 \usepackage{here}

\theoremstyle{plain}
\newtheorem{thm}{Theorem}
\newtheorem{cor}[thm]{Corollary}
\newtheorem{lem}[thm]{Lemma}
\newtheorem{prop}[thm]{Proposition}

\theoremstyle{remark}

\def\hsymbu#1{\smash{\lower1.7ex\hbox{\huge$#1$}}}

\def\rmoveio#1#2{
\setlength{\unitlength}{#1}
\begin{picture}(50,30)
\put(5,0){\line(0,1){30}}

{\allinethickness{.8pt}
\put(10,15){\vector(1,0){13}}
\put(23,15){\vector(-1,0){13}}}

\qbezier(25,0)(25,20)(40,20)
\qbezier(40,20)(45,20)(45,15)
\qbezier(45,15)(45,10)(40,10)
\qbezier(40,10)(35,10)(31,14)
\qbezier(28,17)(25,25)(25,30)

\ifnum#2=2
\put(3,28){\path(0,0)(2,2)(4,0)}
\put(28,28){\path(0,0)(2,2)(4,0)}
\put(38,15){\makebox{${\Huge c_{1}}$}}
\fi

\end{picture}
}

\def\rmoveiio#1#2{
\setlength{\unitlength}{#1}
\begin{picture}(60,30)
\put(5,0){\line(0,1){30}}
\put(15,0){\line(0,1){30}}

{\allinethickness{.8pt}
\put(20,15){\vector(1,0){15}}
\put(35,15){\vector(-1,0){15}}}

\qbezier(40,0)(42,1)(47,3)
\qbezier(52,6)(68,15)(52,24)
\qbezier(47,27)(42,30)(40,30)

\qbezier(60,0)(20,15)(60,30)

\ifnum#2=2
\put(2,27){\path(0,0)(3,3)(6,0)}
\put(12,27){\path(0,0)(3,3)(6,0)}
\put(40,27){\path(0,0)(0,3)(3,3)}
\put(60,27){\path(0,0)(0,3)(-3,3)}
\put(47,16){\makebox{${\Huge c_{1}}$}}
\put(47,10){\makebox{${\Huge c_{2}}$}}
\fi

\ifnum#2=3
\put(2,2){\path(0,0)(3,-3)(6,0)}
\put(12,27){\path(0,0)(3,3)(6,0)}
\put(40,3){\path(0,0)(0,-3)(3,-3)}
\put(60,27){\path(0,0)(0,3)(-3,3)}
\put(47,16){\makebox{${\Huge c_{1}}$}}
\put(47,10){\makebox{${\Huge c_{2}}$}}
\fi

\end{picture}
}
\def\rmoveiiio#1#2{
\setlength{\unitlength}{#1}
\begin{picture}(75,30)
\put(0,0){\line(1,1){15}}
\qbezier(15,15)(20,20)(20,30)

\put(10,0){\line(-1,1){4}}
\qbezier(4,6)(-5,15)(5,25)
\put(5,25){\line(1,1){5}}

\qbezier(20,0)(20,10)(16,14)
\put(14,16){\line(-1,1){8}}
\put(4,26){\line(-1,1){4}}

{\allinethickness{.8pt}
\put(25,15){\vector(1,0){15}}
\put(40,15){\vector(-1,0){15}}}

\qbezier(50,0)(50,10)(55,15)
\put(55,15){\line(1,1){15}}

\put(60,0){\line(1,1){5}}
\qbezier(65,5)(75,15)(66,24)
\put(64,26){\line(-1,1){4}}

\put(70,0){\line(-1,1){4}}
\put(64,6){\line(-1,1){8}}
\qbezier(54,16)(50,20)(50,30)

\ifnum#2=2
\put(0,27){\path(0,0)(0,3)(3,3)}
\put(7,30){\path(0,0)(3,0)(3,-3)}
\put(17,27){\path(0,0)(3,3)(6,0)}

\put(10,23){\makebox{${\Huge c_{1}}$}}
\put(10,3){\makebox{${\Huge c_{2}}$}}
\put(20,13){\makebox{${\Huge c_{3}}$}}

\put(47,27){\path(0,0)(3,3)(6,0)}
\put(60,27){\path(0,0)(0,3)(3,3)}
\put(67,30){\path(0,0)(3,0)(3,-3)}

\put(70,23){\makebox{${\Huge c'_{2}}$}}
\put(70,3){\makebox{${\Huge c'_{1}}$}}
\put(60,13){\makebox{${\Huge c'_{3}}$}}
\fi

\end{picture}
}
\def\rmovevio#1#2{
\setlength{\unitlength}{#1}
\begin{picture}(50,30)
\put(5,0){\line(0,1){30}}

{\allinethickness{.8pt}
\put(10,15){\vector(1,0){15}}
\put(25,15){\vector(-1,0){15}}}

\qbezier(30,0)(30,20)(45,20)
\qbezier(45,20)(50,20)(50,15)
\qbezier(50,15)(50,10)(45,10)
\qbezier(45,10)(30,10)(30,30)

\put(34,15){\circle{5}}

\ifnum#2=2
\put(3,28){\path(0,0)(2,2)(4,0)}
\put(28,28){\path(0,0)(2,2)(4,0)}
\put(38,15){\makebox{${\Huge c_{1}}$}}
\fi

\end{picture}
}

\def\rmoveviio#1#2{
\setlength{\unitlength}{#1}
\begin{picture}(60,30)
\put(5,0){\line(0,1){30}}
\put(15,0){\line(0,1){30}}

{\allinethickness{.8pt}
\put(20,15){\vector(1,0){15}}
\put(35,15){\vector(-1,0){15}}}

\qbezier(40,0)(80,15)(40,30)
\qbezier(60,0)(20,15)(60,30)
\put(50,4){\circle{5}}
\put(50,25){\circle{5}}

\ifnum#2=2
\put(2,27){\path(0,0)(3,3)(6,0)}
\put(12,27){\path(0,0)(3,3)(6,0)}
\put(40,27){\path(0,0)(0,3)(3,3)}
\put(60,27){\path(0,0)(0,3)(-3,3)}
\put(47,15){\makebox{${\Huge c_{1}}$}}
\put(47,10){\makebox{${\Huge c_{2}}$}}
\fi

\ifnum#2=3
\put(2,2){\path(0,0)(3,-3)(6,0)}
\put(12,27){\path(0,0)(3,3)(6,0)}
\put(40,3){\path(0,0)(0,-3)(3,-3)}
\put(60,27){\path(0,0)(0,3)(-3,3)}
\put(47,15){\makebox{${\Huge c_{1}}$}}
\put(47,10){\makebox{${\Huge c_{2}}$}}
\fi

\end{picture}
}

\def\rmoveviiio#1#2{
\setlength{\unitlength}{#1}
\begin{picture}(70,30)
\put(0,0){\line(1,1){15}}
\qbezier(15,15)(20,20)(20,30)

\put(10,0){\line(-1,1){5}}
\qbezier(5,5)(-5,15)(5,25)
\put(5,25){\line(1,1){5}}

\qbezier(20,0)(20,10)(15,15)
\put(15,15){\line(-1,1){15}}

\put(5,5){\circle{5}}
\put(15,15){\circle{5}}
\put(5,25){\circle{5}}

{\allinethickness{.8pt}
\put(28,15){\vector(1,0){14}}
\put(42,15){\vector(-1,0){14}}}

\qbezier(50,0)(50,10)(55,15)
\put(55,15){\line(1,1){15}}

\put(60,0){\line(1,1){5}}
\qbezier(65,5)(75,15)(65,25)
\put(65,25){\line(-1,1){5}}

\put(70,0){\line(-1,1){15}}
\qbezier(55,15)(50,20)(50,30)

\put(65,5){\circle{5}}
\put(55,15){\circle{5}}
\put(65,25){\circle{5}}

\ifnum#2=2
\put(0,27){\path(0,0)(0,3)(3,3)}
\put(7,30){\path(0,0)(3,0)(3,-3)}
\put(17,27){\path(0,0)(3,3)(6,0)}

\put(20,13){\makebox{${\Huge c_{1}}$}}

\put(47,27){\path(0,0)(3,3)(6,0)}
\put(60,27){\path(0,0)(0,3)(3,3)}
\put(67,30){\path(0,0)(3,0)(3,-3)}

\put(60,13){\makebox{${\Huge c'_{1}}$}}
\fi

\end{picture}
}

\def\rmovevivo#1#2{
\setlength{\unitlength}{#1}
\begin{picture}(70,30)
\put(0,0){\line(1,1){15}}
\qbezier(15,15)(20,20)(20,30)

\put(10,0){\line(-1,1){5}}
\qbezier(5,5)(-5,15)(5,25)
\put(5,25){\line(1,1){5}}

\qbezier(20,0)(20,10)(16,14)
\put(14,16){\line(-1,1){14}}

\put(5,5){\circle{5}}
\put(5,25){\circle{5}}

{\allinethickness{.8pt}
\put(28,15){\vector(1,0){14}}
\put(42,15){\vector(-1,0){14}}}

\qbezier(50,0)(50,10)(55,15)
\put(55,15){\line(1,1){15}}

\put(60,0){\line(1,1){5}}
\qbezier(65,5)(75,15)(65,25)
\put(65,25){\line(-1,1){5}}

\put(70,0){\line(-1,1){14}}
\qbezier(54,16)(50,20)(50,30)

\put(65,5){\circle{5}}
\put(65,25){\circle{5}}

\ifnum#2=2
\put(0,27){\path(0,0)(0,3)(3,3)}
\put(7,30){\path(0,0)(3,0)(3,-3)}
\put(17,27){\path(0,0)(3,3)(6,0)}

\put(20,13){\makebox{${\Huge c_{1}}$}}

\put(47,27){\path(0,0)(3,3)(6,0)}
\put(60,27){\path(0,0)(0,3)(3,3)}
\put(67,30){\path(0,0)(3,0)(3,-3)}

\put(60,13){\makebox{${\Huge c'_{1}}$}}
\fi

\end{picture}
}
\def\alexnum#1{
\setlength{\unitlength}{#1}
\begin{picture}(90,20)

\put(5,0){\line(1,1){20}}
\put(25,0){\line(-1,1){9}}
\put(14,11){\line(-1,1){9}}

\put(5,0){\path(0,3)(0,0)(3,0)}
\put(25,0){\path(-3,0)(0,0)(0,3)}

\put(0,15){\makebox{{$i$}}}
\put(0,0){\makebox{{$i$}}}
\put(27,15){\makebox{{$i+1$}}}
\put(27,0){\makebox{{$i+1$}}}

\put(55,0){\line(1,1){9}}
\put(75,0){\line(-1,1){20}}
\put(66,11){\line(1,1){9}}

\put(55,0){\path(0,3)(0,0)(3,0)}
\put(75,0){\path(-3,0)(0,0)(0,3)}

\put(50,15){\makebox{{$i$}}}
\put(50,0){\makebox{{$i$}}}
\put(77,15){\makebox{{$i+1$}}}
\put(77,0){\makebox{{$i+1$}}}

\end{picture}
}
\def\alexnumv#1{
\setlength{\unitlength}{#1}
\begin{picture}(35,20)

\put(5,0){\line(1,1){20}}
\put(25,0){\line(-1,1){20}}
\put(15,10){\circle{4}}

\put(5,0){\path(0,3)(0,0)(3,0)}
\put(25,0){\path(-3,0)(0,0)(0,3)}

\put(0,15){\makebox{{$j$}}}
\put(0,0){\makebox{{$i$}}}
\put(27,15){\makebox{{$i$}}}
\put(27,0){\makebox{{$j$}}}

\end{picture}
}
\def\orientedcutpt#1{
\setlength{\unitlength}{#1}
\begin{picture}(20,20)
\put(5,0){\line(0,1){20}}
\put(5.,8){\path(0,0)(2,2)(-2,2)(0,0)} 

\end{picture}
}
\def\alexnumcp#1{
\setlength{\unitlength}{#1}
\begin{picture}(20,20)
\put(5,0){\line(0,1){20}}
\put(5.,8){\path(0,0)(2,2)(-2,2)(0,0)} 

\put(7,13){\makebox{$i$}}
\put(7,3){\makebox{$i+1$}}
\end{picture}
}
\def\canocutsysi#1#2{
\setlength{\unitlength}{#1}
\begin{picture}(35,20)

\put(5,0){\line(1,1){20}}
\put(25,0){\line(-1,1){20}}
\put(15,10){\circle{4}}

\put(10,5){\path(0,0)(-3,0)(0,-3)(0,0)}
\put(22,3){\path(0,0)(-3,0)(0,3)(0,0)}

\put(5,0){\path(0,3)(0,0)(3,0)}
\put(25,0){\path(-3,0)(0,0)(0,3)}

\ifnum#2=2
\put(0,15){\makebox{{$i$}}}
\put(0,0){\makebox{{$i$}}}
\put(27,15){\makebox{{$i+1$}}}
\put(27,0){\makebox{{$i+1$}}}
\fi

\end{picture}
}
\def\ocpmovei#1{
\setlength{\unitlength}{#1}
\begin{picture}(60,20)
\put(0,0){\line(1,1){20}}
\put(20,0){\line(-1,1){20}}
\put(10,10){\circle{5}}
\put(15,15){\path(0,0)(3,0)(0,3)(0,0)}

{\allinethickness{.8pt}
\put(25,10){\vector(1,0){10}}
\put(35,10){\vector(-1,0){10}}}

\put(40,0){\line(1,1){20}}
\put(60,0){\line(-1,1){20}}
\put(50,10){\circle{5}}
\put(43,3){\path(0,0)(3,0)(0,3)(0,0)}
\end{picture}
}

\def\ocpmoveii#1{
\setlength{\unitlength}{#1}
\begin{picture}(55,20)
\put(5,0){\line(0,1){20}}
\put(5.,5){\path(0,0)(2,2)(-2,2)(0,0)} 
\put(5.,15){\path(0,0)(2,-2)(-2,-2)(0,0)}

\put(10,7){\makebox{{or}}}

\put(23,0){\line(0,1){20}}
\put(23.,7){\path(0,0)(2,-2)(-2,-2)(0,0)} 
\put(23.,13){\path(0,0)(2,2)(-2,2)(0,0)}

{\allinethickness{.8pt}
\put(30,10){\vector(1,0){10}}
\put(40,10){\vector(-1,0){10}}}
\put(50,0){\line(0,1){20}}
\end{picture}
}

\def\ocpmoveiiia#1{
\setlength{\unitlength}{#1}
\begin{picture}(60,20)

\put(0,0){\line(1,1){20}}
\put(20,0){\line(-1,1){9}}
\put(9,11){\line(-1,1){9}}

\put(3,3){\path(0,0)(3,0)(0,3)(0,0)}
\put(17,3){\path(0,0)(0,3)(-3,0)(0,0)}
\put(17,17){\path(0,0)(-3,0)(0,-3)(0,0)}
\put(3,17){\path(0,0)(0,-3)(3,0)(0,0)}

{\allinethickness{.8pt}
\put(25,10){\vector(1,0){10}}
\put(35,10){\vector(-1,0){10}}}

\put(40,0){\line(1,1){20}}
\put(60,0){\line(-1,1){9}}
\put(49,11){\line(-1,1){9}}

\end{picture}
}
\def\ocpmoveiiib#1{
\setlength{\unitlength}{#1}
\begin{picture}(60,20)

\put(0,0){\line(1,1){20}}
\put(20,0){\line(-1,1){9}}
\put(9,11){\line(-1,1){9}}

\put(5,5){\path(0,0)(-3,0)(0,-3)(0,0)}
\put(15,5){\path(0,0)(0,-3)(3,0)(0,0)}
\put(15,15){\path(0,0)(3,0)(0,3)(0,0)}
\put(5,15){\path(0,0)(0,3)(-3,0)(0,0)}

{\allinethickness{.8pt}
\put(25,10){\vector(1,0){10}}
\put(35,10){\vector(-1,0){10}}}

\put(40,0){\line(1,1){20}}
\put(60,0){\line(-1,1){9}}
\put(49,11){\line(-1,1){9}}

\end{picture}
}

\begin{document}
\title[Cyclic coverings of  virtual link diagrams]
{Cyclic coverings of  virtual link diagrams}
\author{Naoko Kamada} 
\thanks{This work was supported by JSPS KAKENHI Grant Number 15K04879.}
\address{ Graduate School of Natural Sciences,  Nagoya City University\\ 
1 Yamanohata, Mizuho-cho, Mizuho-ku, Nagoya, Aichi 467-8501 Japan
}

\date{}

\begin{abstract} 
A virtual link diagram is called mod $m$ almost classical if it admits an Alexander numbering valued in integers modulo $m$, and a virtual link is called mod $m$ almost classical 
if it has a mod $m$ almost classical diagram as a representative. 
In  this paper, we introduce a method of constructing a mod $m$ almost classical virtual link diagram from a given virtual link diagram, which we call an $m$-fold cyclic covering diagram.   The main result is that 
$m$-fold cyclic covering diagrams obtained from two equivalent virtual link diagrams are equivalent. Thus we have a well-defined map from the set of virtual links to the set of mod $m$ almost classical virtual links.   Some applications are also given.  

\end{abstract}
\maketitle

\section{Introduction}

Virtual links, introduced by L. H. Kauffman  \cite{rkauD},   correspond to  abstract links \cite{rkk} and 
stable equivalence classes of  links in thickened surfaces \cite{rCKS,rkk}. 
A virtual link diagram is called {\it almost classical}  if it admits an Alexander numbering (cf. \cite{rsilver}), and it is called 
 {\it mod $m$ almost classical}  
if it admits an Alexander numbering in $\mathbb{Z}_m$ (cf. \cite{rboden}). 
A virtual link is called almost classical (resp. mod $m$ almost classical)  if it has an almost classical (resp. mod $m$ almost classical) virtual link 
diagram as a representative. 
Every classical link diagram is almost classical, and every almost classical virtual link diagram is mod $m$ almost classical. 
A virtual link diagram is checkerboard colorable if and only if it is mod 2 almost classical. 
It is known that Jones polynomials of mod $2$ almost classical virtual links have a property that Jones polynomials of classical links have (\cite{rkn0, rkn2004}).  
Alexander polynomials for mod $m$ almost classical virtual links can be defined in a similar way to those for almost classical link diagrams \cite{rboden}. 

In this paper, we introduce the notion of an oriented cut point and a cut system for a virtual link diagram, which is an extension of (unorieted) cut points introduced by H.~Dye in \cite{rdye2, rdye2017}.  
For any pair $(D,P)$ of a virtual link diagram $D$ and a cut system $P$, we construct a virtual link diagram $\varphi_m(D,P)$ which is mod $m$ almost classical. We call it an {\it $m$-fold cyclic covering (virtual link) diagram}  of $(D,P)$.   

It turns out that the strong equivalence class of $\varphi_m(D,P)$ does not depend on $P$, namely, for any cut systems $P$ and $P'$ of the same virtual link diagram $D$, $\varphi_m(D,P)$ and $\varphi_m(D,P')$ are strongly equivalent (Lemma~\ref{lem2}).  
Our main theorem (Theorem~\ref{thm2}) states that if virtual link diagrams $D$ and $D'$ are equivalent, then $\varphi_m(D,P)$ and $\varphi_m(D', P')$ are equivalent.  
Thus, we obtain a well-defined map from the set of virtual links to the set of mod $m$ almost classical virtual links. 

As an application, we demonstrate how Theorem~\ref{thm2} is used to show that two virtual link diagrams are not equivalent. Theorem~\ref{thm2} implies Theorem~\ref{thm:appli} that if $\varphi_m(D,P)$ is not equivalent to a disjoint union of $m$ copies of $D$ itself then $D$ is never equivalent to a mod $m$ virtual link diagram, i.e., the virtual link represented by $D$ is not mod $m$ almost classical.  

This paper is organized as follows: In Section~\ref{secnormal} we recall virtual link diagrams and Alexander numberings, and introduce the notions of an oriented cut point and a cut system.  
In~Section~\ref{sect:CyclicCover1} we give a method of construction of $\varphi_m(D,P)$.  It is shown that $\varphi_m(D,P)$ is a mod $m$ almost classical virtual link diagram.   
In Section~\ref{sect:main}, main results,  
Lemma~\ref{lem2} and Theorem~\ref{thm2}, are introduced and proved.  
In Section~\ref{sect:CyclicCover2} we give an alternative method of constructing cyclic covering virtual link diagrams. 
In Section~\ref{application} we show some applications.

\section{Alexander numberings and cut systems}\label{secnormal} 

In this section we recall virtual link diagrams and Alexander numberings, and introduce the notions of an oriented cut point and a cut system, which are used for our construction of cyclic covering diagrams.  

A {\it virtual link diagram\/} is a generically immersed, closed and oriented 1-manifold in $\mathbb{R}^2$ with information of positive, negative or virtual crossing, on each double point.  Here a {\it virtual crossing\/} means an encircled double point without over-under information \cite{rkauD}. 
%
%
{\it Generalized Reidemeister moves} are the local moves depicted in Figure~ \ref{fgmoves}: The 3 moves on the top are {\it (classical) Reidemeister moves} and the 4 moves on the bottom are so-called {\it virtual Reidemeister moves}. 
Two virtual link diagrams $D$ and $D'$ are said to be {\it equivalent} (resp. {\it strongly equivalent}) 
 if they are related by a finite sequence of generalized Reidemeister moves (resp. virtual Reidemeister moves) and isotopies of $\mathbb{R}^2$.  
A {\it virtual link\/}  (resp. a {\it pre-virtual link}) is an equivalence class (resp. a strong equivalence class) 
of virtual link diagrams. 

\vspace{0.3cm}
\begin{figure}[h]
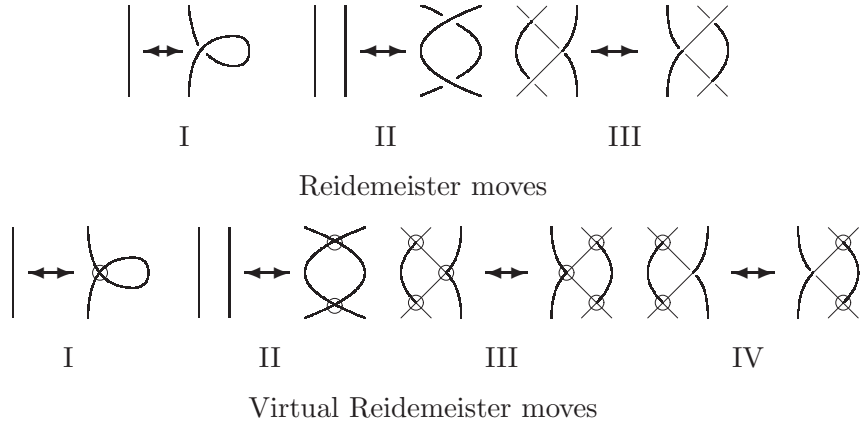

\centerline{
\begin{tabular}{ccc}
\rmoveio{.4mm}{1}&\rmoveiio{.4mm}{1}&\rmoveiiio{.4mm}{1}\\
I&II&III\\
\multicolumn{3}{c}{Reidemeister moves}
\end{tabular}}
\vspace{0.2cm}
\centerline{
\begin{tabular}{cccc}
\rmovevio{.4mm}{1}&\rmoveviio{.4mm}{1}&\rmoveviiio{.4mm}{1}&\rmovevivo{.4mm}{1}\\
I&II&III&IV\\
\multicolumn{4}{c}{Virtual Reidemeister moves}
\end{tabular}}
\caption{Generalized Reidemeister moves}\label{fgmoves}
\end{figure}

A {\it virtual path} of a virtual link diagram $D$ means a path (possibly a loop) on $D$ on which there are no classical crossings. 
A virtual link diagram $D'$ is said to be obtained from $D$ by 
a {\it detour move} if $D'$ is obtained by replacing a virtual path of $D$ with 
a path which is a virtual path of $D'$.  Two diagrams $D$ and $D'$ are strongly equivalent if and only if they are related by a finite sequence of detour moves and isotopies of $\mathbb{R}^2$ 
 (cf. \cite{rkk, rkauD}).

Let $D$ be a virtual link diagram. 
{\it A semi-arc} of $D$ is a virtual path which is an immersed arc between two classical crossings of $D$ or an immersed  loop.  
Let $m$ be a non-negative integer. 
An {\it Alexander numbering} (resp. a {\it mod $m$ Alexander numbering}) 
of $D$ is an assignment of a number  of $\mathbb{Z}$ (resp. $\mathbb{Z}_m$) 
to each semi-arc of $D$ such that the numbers of 4 semi-arcs  around each classical crossing are as  shown in Figure~\ref{fgalexnum} for some $i \in \mathbb{Z}$ (resp. $i \in \mathbb{Z}_m$).

\vspace{0.3cm}
\begin{figure}[h]
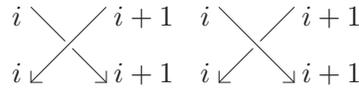

\centerline{
\alexnum{.5mm}
}
\caption{Alexander numbering}\label{fgalexnum}
\end{figure}

Note that the numbers assigned to semi-arcs around a virtual  crossing  is depicted as in Figure~\ref{fgalexnumv}.
\begin{figure}[h]
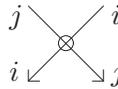

\centerline{
\alexnumv{.5mm}
}
\caption{Alexander numbering around a virtual crossing}\label{fgalexnumv}
\end{figure}

An example of an Alexander numbering is depicted in Figure~\ref{fgexAlexN1}.  A classical link diagram always admits an Alexander numbering. 
\begin{figure}[h]
\centerline{
\includegraphics[width=2.5cm]{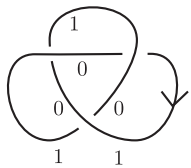}
}
\caption{An Alexander numbering of a classical link diagram}\label{fgexAlexN1}
\end{figure}

Not every virtual link diagram admits an Alexander numbering.   The virtual link diagram depicted in Figure~\ref{fgexAlexN2} (i) does not admit an Alexander numbering,  and the virtual link diagram 
in Figure~\ref{fgexAlexN2} (ii) does. 

\begin{figure}[h]
\centerline{
\includegraphics[width=7.cm]{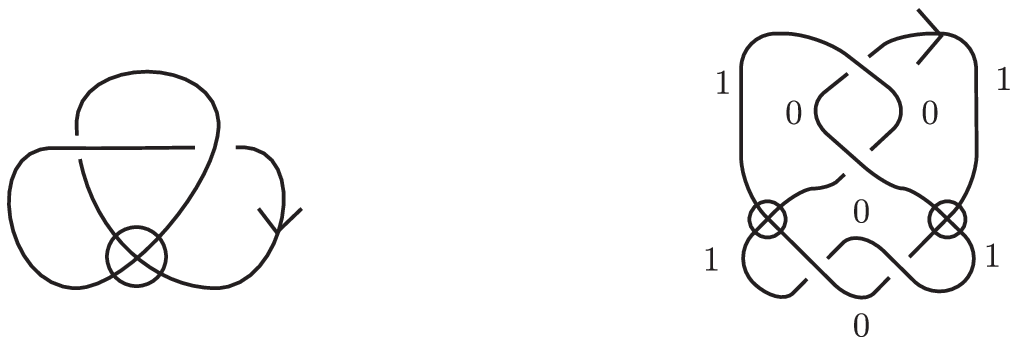}
}
\centerline{(i)\hspace{4.5cm}(ii)}
\caption{Virtual link diagrams which does/does not admit an Alexander numbering}\label{fgexAlexN2}
\end{figure}

Figure~\ref{fgexMAlexN1} shows an example of a mod $3$ Alexander numbering, which is not an Alexander numbering.

\begin{figure}[h]
\centerline{
\includegraphics[width=4.5cm]{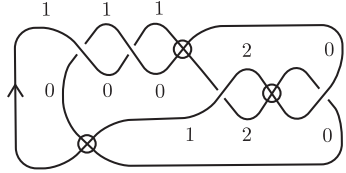}
}
\caption{An mod $3$ Alexander numbering of a virtual link diagram}\label{fgexMAlexN1}
\end{figure}

A virtual link diagram  is {\it almost classical} (resp. {\it mod $m$ almost classical\/}) 
if it admits an Alexander numbering (resp. a mod $m$ Alexander numbering). 
A virtual link $L$ is {\it almost classical\/} (resp. {\it mod $m$ almost classical})  
if there is an almost classical (resp. mod $m$ almost classical) virtual link diagram of $L$.

H. Boden, R.Gaudreau, E. Harper, A. Nicas, L. White \cite{rboden} studied mod $m$ almost classical virtual links. 
By definition, 
any almost classical virtual link diagram is mod $m$ almost classical.  
A virtual link diagram is checkerboard colorable  if and only if it is mod $2$ almost classical.
It is shown in \cite{rboden} that for a mod $m$ almost classical virtual knot $K$, if $D$ is a minimal virtual knot diagram of $K$, then $D$ is mod $m$ almost classical.

H. Dye introduced the notion of a cut point \cite{rdye2}, which is an \lq unoriented\rq \, cut point in our sense.  
The author \cite{rkn2004} generalized the Kauffman-Murasugi-Thistlethwaite theorem (\cite{rkau1987, rmurasugi, rthis}) on the span of the Jones polynomial of a classical link to checkerboard colorable and proper virtual links. 
Using cut points, H. Dye \cite{rdye2017} further extended this result to virtual link diagrams that are not checkerboard colorable.

Using (unoriented) cut points, 
the author constructed in  \cite{rkn2, rkn3} a map from the set of virtual links to the set of checkerboard colorable virtual links, i.e., the set of mod $2$ almost classical virtual links.  In this paper, we generalize this to the mod $m$ case. 

An {\it oriented cut point} or simply a {\it cut point} is a point on an arc at which a local orientation of the arc is given. In this paper we denote it by a small triangle on the arc as in Figure~\ref{fgorientedcutpt}.  
Whenever cut points on a virtual link diagram are discussed, we assume that they are on semi-arcs of the diagram avoiding  crossings.  An oriented cut point is called  {\it coherent} (resp.  {\it incoherent}) if the local orientation indicated by the cut point  is coherent  (resp. incoherent) to the orientation of the virtual link diagram.

\vspace{0.3cm} 
\begin{figure}[h]
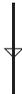

\centerline{
\orientedcutpt{.6mm}
}
\caption{An oriented cut point on an arc}\label{fgorientedcutpt}
\end{figure}

Let $D$ be  a virtual link diagram and $P$ a set of oriented cut points of $D$. 
We say that $P$ is a {\it cut system}  if 
$D$ admits an Alexander numbering such that 
at each oriented cut point,  the number increases by one in the direction of the 
oriented cut point (Figure~\ref{fgalexnumorict}).   
Such an Alexander numbering is called an {\it Alexander numbering of a virtual link diagram with a cut system}. 
See Figure~\ref{fgExcutPtAlex} for examples. 

\vspace{0.3cm} 
\begin{figure}[h]
\centerline{
\alexnum{.5mm}\hspace{.6cm}\alexnumv{.5mm}\hspace{.6cm}\alexnumcp{.5mm}
}
\caption{Alexander numbering of a virtual link diagram with a cut system}\label{fgalexnumorict}
\end{figure}
\vspace{0.3cm} 
\begin{figure}[h]
\centerline{
\includegraphics[width=8.5cm]{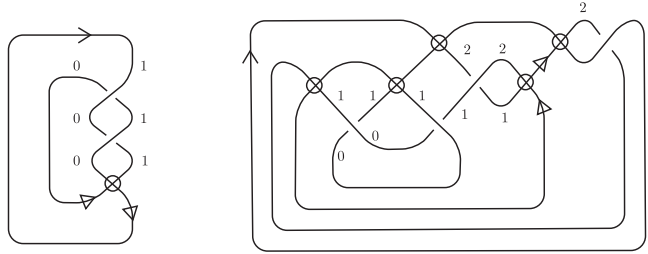}
}
\caption{Alexander numberings of virtual link diagrams with cut systems}\label{fgExcutPtAlex}
\end{figure}

For a virtual link diagram $D$ with a cut system $P$, let ${\rm Arc}(D,P)$ be the set of arcs (or loops) obtained from semi-arcs of $D$ by cutting along $P$.  (If there is a semi-arc of $D$ which is a loop and has no cut points of $P$, then ${\rm Arc}(D,P)$ has the loop as an element.)  An Alexander numbering of $D$ with $P$ is regarded as a map from ${\rm Arc}(D,P)$ to ${\mathbb Z}$.  For a semi-arc $a$ of $D$ not being a loop, we denote by $a_-$ (resp. $a_+$) the arc of ${\rm Arc}(D,P)$ which contains the starting point (resp. the terminal point) of $a$.  

\begin{lem}\label{lem:CutNumberOnArc}
Let $f : {\rm Arc}(D,P) \to {\mathbb Z}$ be an Alexander numbering of a virtual link diagram $D$ with a cut system $P$.  
\begin{itemize}
\item[(1)] For any semi-arc $a$ of $D$ not being a loop, $f(a_+) - f(a_-)$ is the number of coherent cut points minus the number of incoherent cut points of $P$ appearing on $a$. 
\item[(2)] For any semi-arc $a$ of $D$ being loop, the number of coherent cut points minus the number of incoherent cut points of $P$ appearing on  $a$ is $0$.  
\end{itemize}
\end{lem}

\begin{proof} 
It is obvious, since when we move along $a$ from $a_-$ to $a_+$, 
the numbers assigned by $f$ changes by $+1$ (resp. $-1$) at each coherent (resp. incoherent) cut point.  
\end{proof}

A {\it canonical cut system} of a virtual link diagram is a cut system which is obtained by introducing 
 two oriented cut points as in Figure~\ref{fg:canocutsys} around each classical crossing. 
It is really a cut system and an Alexander numbering  looks as in Figure~\ref{fg:canocutsys} around each virtual crossing. 
 
\vspace{0.3cm}
\begin{figure}[h]
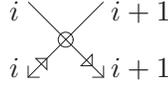

\centerline{
\canocutsysi{.5mm}{2}}
\caption{Canonical cut system}\label{fg:canocutsys}
\end{figure}
%

%

The local transformations of oriented cut points depicted in Figure~\ref{fgOrientedCutmv} are called 
 {\it oriented cut point moves}.  For a virtual link diagram with a cut system, the result by an oriented cut point move  is also a cut system of the same virtual link diagram.   
Note that the move III$'$ depicted in Figure~\ref{fgOrientedCutmv} is obtained from the move III modulo the moves II. 

\vspace{0.3cm}
\begin{figure}[h]
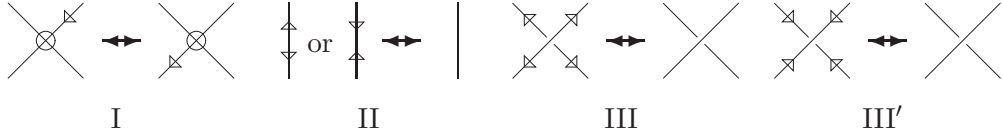

\centerline{
\begin{tabular}{cccc}
\ocpmovei{.5mm}&\ocpmoveii{.5mm}&\ocpmoveiiia{.5mm}&\ocpmoveiiib{.5mm}\\
I&II&III&III$'$\\
\end{tabular}
}
\caption{Oriented cut moves}\label{fgOrientedCutmv}
\end{figure}

\begin{thm} \label{thm:cutpointmove}
Two cut systems of the same virtual link diagram are related by a sequence of oriented cut point moves. 
\end{thm}

\begin{proof} 
Let $P$ and $P'$ be cut systems of a virtual link diagram $D$. 
Let $f$ (resp. $f'$) be an Alexander numbering of $D$ with cut system $P$ (resp. $P')$. 
Applying a finite number of oriented cut point moves III to $P$, we obtain a cut system $P''$ and an Alexander numbering $f''$ such that 
the numberings of 4 edges around each classical crossing  are as same as those of $f'$. 
By Lemma~\ref{lem:CutNumberOnArc}, we see that for any semi-arc $a$ of $D$,  
the number of coherent cut points minus the number of incoherent cut points of $P''$ appearing on $a$  is equal to that of $P'$.  Thus, by using oriented cut point moves I and II, $P''$ can be transformed to $P'$.  
\end{proof}

\begin{cor}\label{cor:numbers}
Let $D$ be a virtual link diagram and let $P$ be a cut system of $D$. The number of coherent cut points of $P$ equals that of incoherent cut points of $P$. 
\end{cor}

\begin{proof}
The canonical cut system for $D$ has the property that the number of coherent cut points  equals that of incoherent cut points.  Since each oriented cut point move preserves this property, by Theorem~\ref{thm:cutpointmove} we see that any cut system has the property.  
\end{proof}

\section{Cyclic coverings of  virtual link diagrams}\label{sect:CyclicCover1}

In this section,  we introduce a method of constructing a mod $m$ almost classical virtual link diagram $\varphi_m(D,P)$, which is determined up to strong equivalence,  from a virtual link diagram $D$ with a cut system $P$. 

\vspace{0.3cm}
We denote by a pair $(D,P)$ a virtual link diagram $D$ with a cut system $P$. 
Moving $(D,P)$ slightly by an isotopy of ${\mathbb R}^2$, we assume that each cut point $p$ of $P$ is on a horizontal line $\ell(p)$ in ${\mathbb R}^2$ such that 
$\ell(p)$ intersects $D$ transversely avoiding all crossings of $D$ and $p$ is a unique cut point of $P$ on $\ell(p)$. 
Let $(D^0, P^0), (D^1,  P^1), \dots, (D^{m-1}, P^{m-1})$ be $m$ parallel copies of $(D, P)$ with $(D^0, P^0)=(D,P)$ obtained from $(D,P)$ by sliding along the $x$-axis  such that they appear from  left to right in this order.  For each cut point $p \in P$, we denote by $p^k$ the copy of $p$ in $P^k$ for $k \in \{0, \dots, m-1\}$. See Figure~\ref{fgAlexcov1p} for an example. (The Alexander numberings in the figure are used later.) 

\begin{figure}[h]
\centerline{\includegraphics[width=9cm]{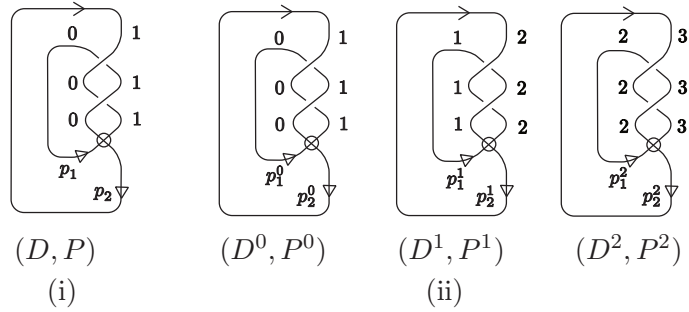}}
\centerline{$(D, P)$ \hspace{1.5cm}$(D^0, P^0)$ \hspace{.8cm}$(D^1, P^1)$ \hspace{.8cm}$(D^{2}, P^{2})$}
\centerline{(i)\hspace{4.5cm} (ii)\phantom{MMMMMMM}}
\caption{$3$ copies of a virtual link diagram with cut system}\label{fgAlexcov1p}
\end{figure}

For each $p \in P$, let $N(\ell(p))$ be a regular neighborhood of the horizontal line $\ell(p)$ in ${\mathbb R}^2$. 
In  Figure~\ref{fgAlexcov2p}, $N(\ell(p))$ is the part between two dotted lines parallel to $\ell(p)$.  
The diagram $\cup_{k=0}^{m-1}D^k$ looks locally near $N(\ell(p))$ 
as in the upper part of Figure~\ref{fgAlexcov2p}.  Replace it as in the lower part of the figure for every $p \in P$,  where the doted arc drawn in the very bottom of the figure means a virtual path and we may put it anyplace  as long as it contains only virtual crossings.  The virtual link diagram obtained this way is denoted by $\varphi_m(D, P)$ and is called an {\it $m$-fold cyclic covering (virtual link) diagram} of $(D, P)$. 

In the early stage of this construction, we modified $(D,P)$ by an isotopy of ${\mathbb R}^2$.  When we modify $(D,P)$ differently, the diagram $\varphi_m(D, P)$ may change. However it is preserved up to strong equivalence. 
Although this fact can be seen by observing how the diagram $\varphi_m(D, P)$ changes by a modification of $(D,P)$, we will show it in a more general situation as Theorem~\ref{thm:General} in Section~\ref{sect:CyclicCover2}. 

\begin{figure}[h]
\centerline{\includegraphics[width=15cm]{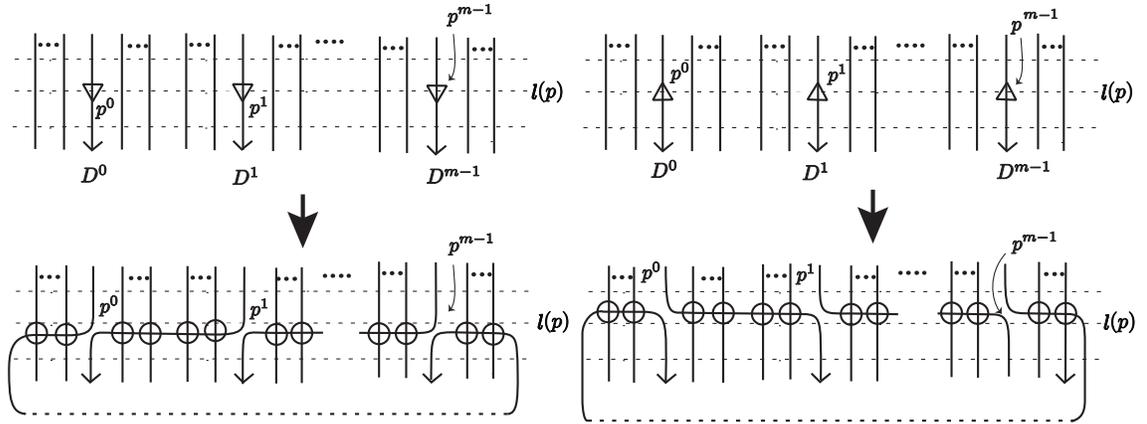}}
\caption{Construction of $m$ cyclic covering}\label{fgAlexcov2p}
\end{figure}

For example, for $(D, P)$  depicted in Figure~\ref{fgAlexcov1p} (i), a $3$-fold cyclic covering virtual link diagram $\varphi_3(D, P)$ is shown in Figure~\ref{fgExAlexcov1p}.

\begin{figure}[h]
\centerline{
\includegraphics[width=8cm]{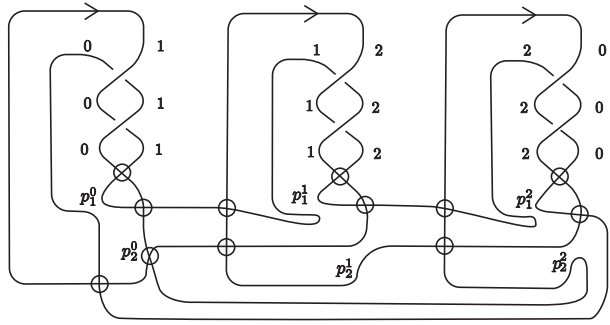}}
\caption{A $3$-fold cyclic covering virtual link diagram of $(D,P)$}\label{fgExAlexcov1p}
\end{figure}

\begin{prop}\label{thm1}
For a virtual link diagram $D$ with a cut system $P$, an $m$-fold cyclic covering virtual link diagram $\varphi_m(D, P)$ is mod $m$ almost classical.
\end{prop}
\begin{proof}
Let $f$ be an Alexander numbering of $(D,P)$.  
For each $k \in \{0, \dots, m-1\}$, let $f^k$ denote the Alexander numbering of $(D^k, P^k)$ obtained from $f$ by shifting $k$.  As shown in Figure~\ref{fgAlexcovprfp}, the Alexander numberings $f_0, \dots, f_{m-1}$ induce 
a mod $m$ Alexander numbering  of $\varphi_m(D, P)$.   For example, see Figures~\ref{fgAlexcov1p} and~\ref{fgExAlexcov1p}. 
\end{proof}
\begin{figure}[h]
\centerline{
\includegraphics[width=8cm]{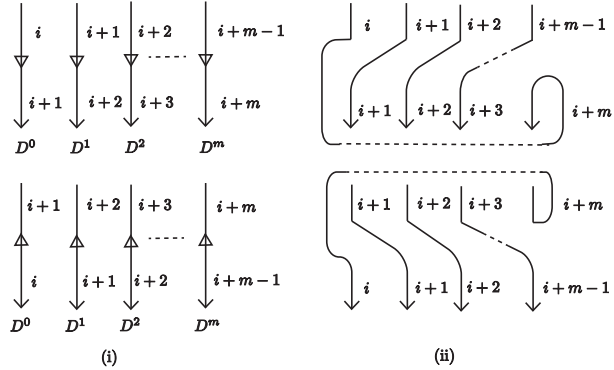}}
\caption{Alexander numbering of a cyclic covering virtual link diagram}\label{fgAlexcovprfp}
\end{figure}

%
%

\section{The main theorem} \label{sect:main}

In Section~\ref{sect:CyclicCover1}, we introduced an $m$-fold cyclic covering diagram $\varphi_m(D,P)$ for a virtual link diagram $D$ with  a cut system $P$.  In this section, 
we first show that $\varphi_m(D,P)$, up to strong equivalence, does not depend on $P$ (Lemma~\ref{lem2}).  Hence we may denote it by $\varphi_m(D)$.  
Our main theorem is that if $D$ and $D'$ are equivalent then 
$\varphi_m(D,P)$ and $\varphi_m(D',P')$ are equivalent (Theorem~\ref{thm2}).  This implies that we have a map $\varphi_m$ from the set of virtual links to the set of mod $m$ almost classical virtual links.

\begin{lem}\label{lem2}
Let $D$ be a virtual  link diagram, and let $P_1$ and  $P_2$ be cut systems of $D$.
Then $\varphi_m(D,P_1)$ and $\varphi_m(D,P_2)$ are strongly equivalent.
\end{lem}

\begin{proof}
Suppose that $P_1$ and $P_2$ are as in the left part of Figure~\ref{fgcmovprf}. Then $\varphi_m(D, P_1)$ and $\varphi_m(D, P_2)$ are as in the right part of the figure, which are   
related by detour moves.  The other cases of oriented cut moves are shown by a similar argument. 
\end{proof}

\begin{figure}[h]
\centerline{
 \includegraphics[width=14cm]{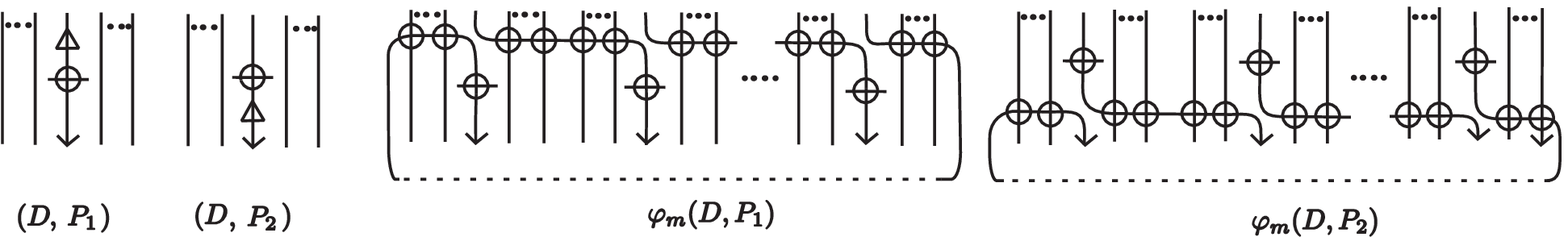}
}
\centerline{(i)}

\vspace{.2cm}
\centerline{
 \includegraphics[width=14cm]{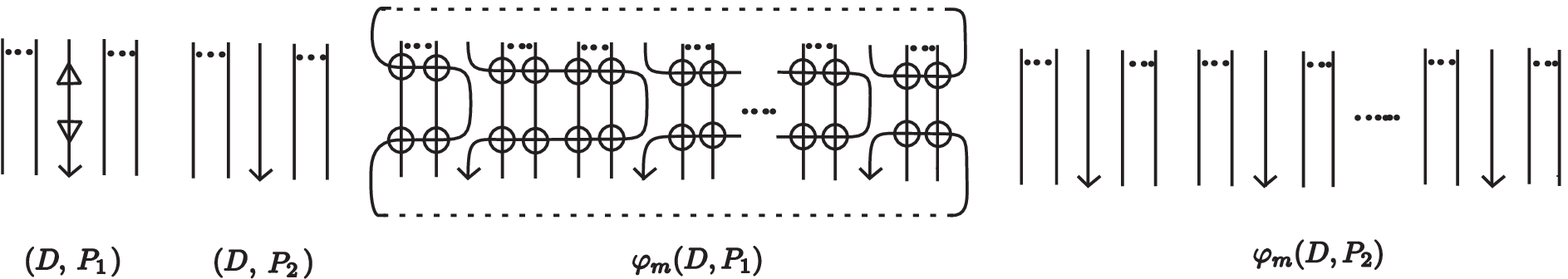}
}
\centerline{(ii)}

\vspace{.2cm}
\centerline{
 \includegraphics[width=14cm]{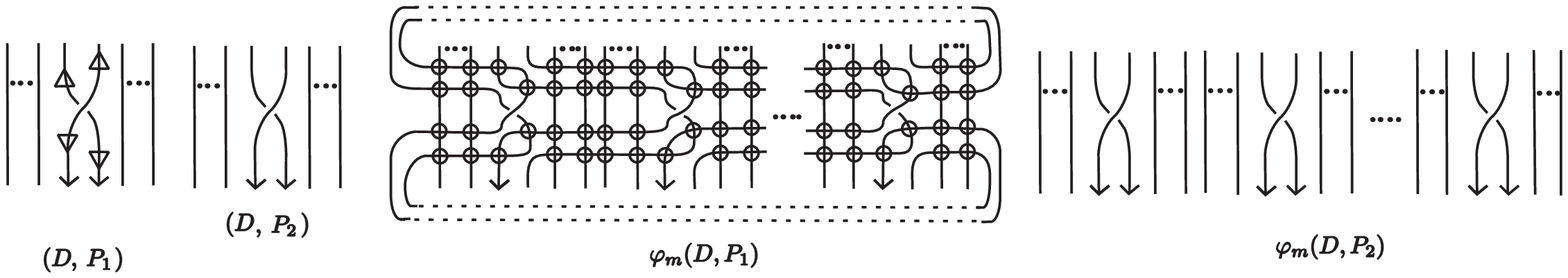}
}
\centerline{(iii)}

\vspace{.2cm}
\caption{Results by an oriented cut point move}\label{fgcmovprf}
\end{figure}

The following is our main theorem. It implies that we have a map $\varphi_m$ from the set of virtual links to the set of mod $m$ almost classical virtual links. 
\begin{thm}\label{thm2}
Let $(D,P)$ and $(D',P')$ be virtual link diagrams with cut systems. If $D$ and $D'$ are equivalent, 
then   $\varphi_m(D,P)$ and $\varphi_m(D',P')$ are equivalent. 
\end{thm}

\begin{proof}
By Lemma~\ref{lem2}, it is sufficient to consider the case that $P$ and $P'$ are canonical cut systems.  

If $D'$ is related to $D$ by one of Reidemeister moves, then $\varphi_m(D, P)$ and $\varphi_m(D',P')$ are related by $m$ Reidemeister moves, which are copies of the original Reidemsiter moves.  

Suppose that $D'$ is related to $D$  by a virtual Reidemeister move I  (resp. II) as 
 in Figure~\ref{fgovprfv} (i) (resp. (ii)).  
 Let $P_*$ be the cut system obtained from $P$ by cut point moves I and II as in the figure. 
By Lemme~\ref{lem2},  $\varphi_m(D,P)$ and $\varphi_m(D,P_*)$ are equivalent.
On the other hand $\varphi_m(D',P')$ and $\varphi_m(D,P_*)$ are related by $m$ virtual Reidemeister moves I (resp. II). Thus $\varphi_m(D,P)$ and  $\varphi_m(D',P')$ are equivalent.

Suppose that  $D'$ is related to $D$  by a virtual Reidemeister move III as in Figure~\ref{fgovprfv} (iii).   
Let $P_*$ (resp. $P'_*$) be the cut system obtained from $P$ (resp. $P'$) by cut point moves as in the figure. 
By Lemme~\ref{lem2},  $\varphi_m(D,P)$  (resp. $\varphi_m(D',P')$ ) and $\varphi_m(D,P_*)$ (resp. $\varphi_m(D',P'_*)$) are equivalent.
On the other hand, $\varphi_m(D,P_*)$ and $\varphi_m(D',P'_*)$ are  related by $m$ virtual Reidemeister moves III. Thus $\varphi_m(D,P)$ and  $\varphi_m(D',P')$ are equivalent.

Suppose that  $D'$ is related to $D$  by a virtual Reidemeister move IV as in Figure~\ref{fgovprfv}  (iv). 
Let $P_*$ (resp. $P'_*$) be the cut system obtained from $P$ (resp. $P'$) by cut point moves as in the figure. By Lemme~\ref{lem2},  $\varphi_m(D,P)$  (resp. $\varphi_m(D',P')$ ) and $\varphi_m(D,P_*)$ (resp. $\varphi_m(D',P'_*)$) are equivalent.
On the other hand, 
$\varphi_m(D,P_*)$  and $\varphi_m(D',P'_*)$ are equivalent by $m$ virtual Reidemeister moves IV. Thus $\varphi_m(D, P)$ and  $\varphi_m(D', P')$ are equivalent. 
The other cases where the orientations of virtual link diagrams are different are shown by a similar argument. 
 \vspace{0.3cm}
\begin{figure}[h]
\centering{
\begin{tabular}{ccc}
\includegraphics[width=4.cm]{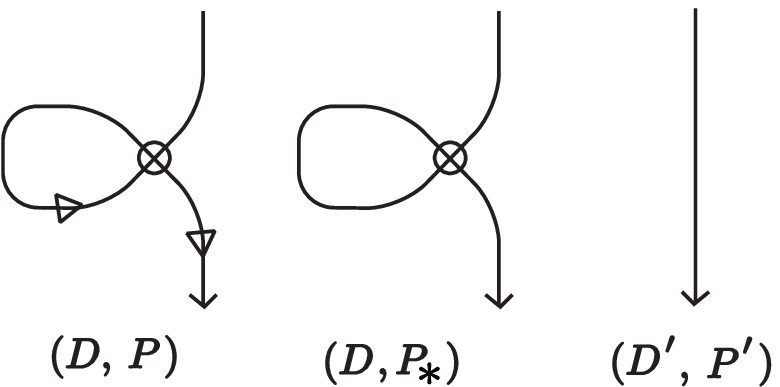}&\phantom{MMM}&\includegraphics[width=8.cm]{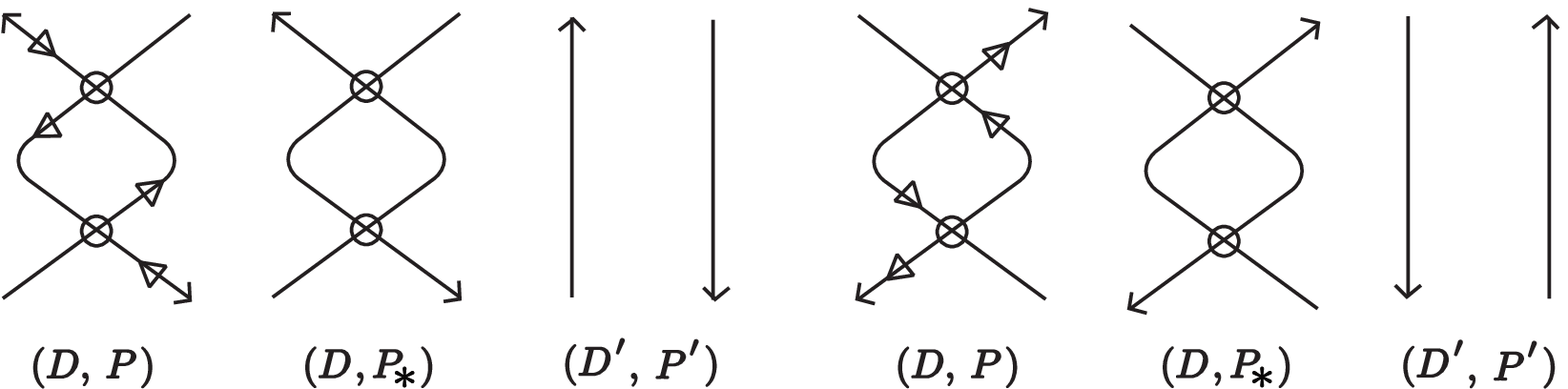}\\
{\small Virtual Reidemeister move I }&\phantom{MMM}&{\small Virtual Reidemeister move II}\\
(i)&\phantom{MMM}&(ii)\\
\end{tabular}
\begin{tabular}{ccc}
\includegraphics[width=6.cm]{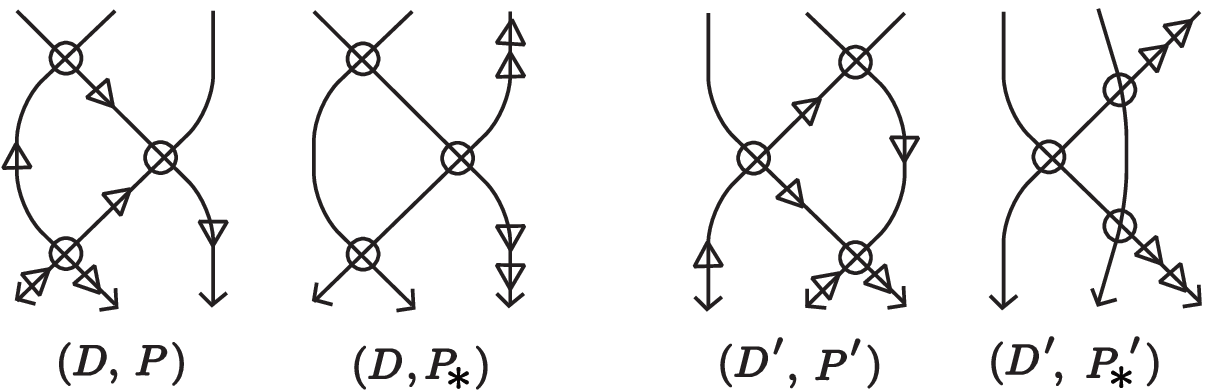}&\phantom{MMM}&\includegraphics[width=6cm]{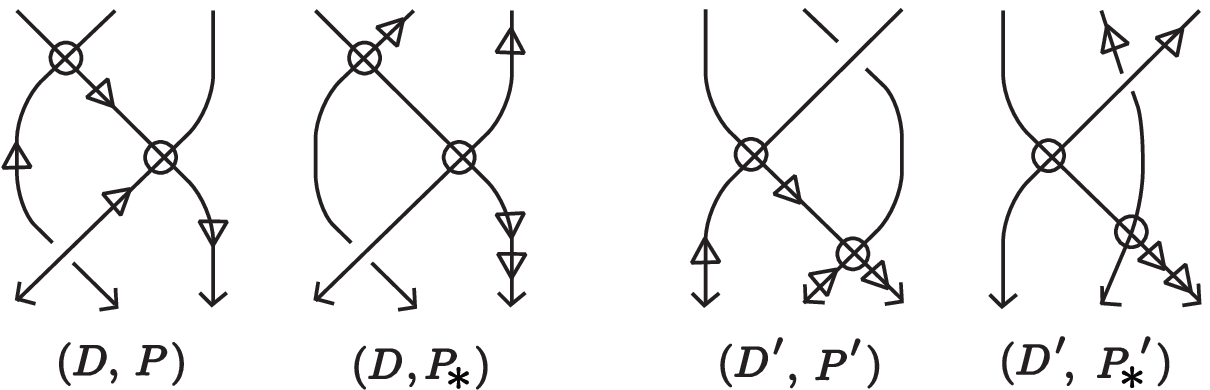}\\
{\small Virtual Reidemeister move III}&\phantom{MMM}&{\small Virtual Reidemeister move IV}\\
(iii)&\phantom{MMM}&(iv)\\
\end{tabular}}
\caption{Diagrams related by a virtual Reidemeister move}\label{fgovprfv}
\end{figure}
%
\end{proof}

\section{An alternative construction of cyclic covering virtual link diagrams}
\label{sect:CyclicCover2}

In this section, we introduce two methods of constructing cyclic covering virtual link diagrams. 
The first one is a more general method, denoted by $\varphi_m^0(D,P)$,  including the method introduced in Section~\ref{sect:CyclicCover1} as a special case.  The second one is a method which is also a special case of the first one.  
The reader who does not need it might skip this section. 

In the construction of $\varphi_m(D,P)$ introduced in Section~\ref{sect:CyclicCover1}, we first modified $(D,P)$ so that 
each horizontal line $\ell(p)$ through $p \in P$ intersects $D$ transversely avoiding the crossings of $D$ and the other cut points of $P$, and then we considered $m$ parallel copies of $(D,P)$.  However, we may define $\varphi_m(D,P)$ without this procedure.  

Let $(D,P)$ be a virtual link diagram with a cut system.  
Let $(D^k, P^k)$, $k=0, \dots, m-1$, be virtual link diagrams with cut systems such that 
each $(D^k, P^k)$ is a copy of $(D,P)$ and that the intersection of $D^k$ and $D^{k'}$ for $k \neq k'$ is empty or consists of virtual crossings. 
(Furthermore, we may weaken the assumption that $(D^k, P^k)$ is a copy of $(D,P)$ so that $(D^k, P^k)$ is isotopic to $(D,P)$ by an isotopy of ${\mathbb R}^2$ or even that $(D^k, P^k)$ is strongly equivalent to $(D,P)$.) 
For each $p \in P$, let $N(p)$ be a regular neighborhood of $p$ in $D$, which is a small arc on $D$ containing $p$. Let $p_-$ and $p_+$ be the endpoints of $N(p)$ such that the orientation of the virtual link diagram restricted to $N(p)$ is from $p_-$ to $p_+$.    
For each $k \in \{0, \dots, m-1\} = {\mathbb Z}_m$, let $p^k$, $N(p^k)$, $p^k_-$ and $p^k_+$ be the corresponding copy of 
$p$, $N(p)$, $p_-$ and $p_+$ in $D^k$.   
Remove $N(p^k)$ for all $p \in P$ and $k \in \{0, \dots, m-1\}$ from the diagram $\cup_{k=0}^{m-1}D^k$ and, for each $p \in P$ and $k \in \{0, \dots, m-1\}$, connect the endpoint $p^k_-$ to $p^{k-\epsilon(p)}_+$ by any virtual path.   Here $\epsilon(p)$ is $+1$ (resp. $-1$) if $p$ is coherent (resp. incoherent).  
We denote by $\varphi_m^0(D,P)$ a virtual link diagram obtained this way.  

Consider an Alexander numbering $f$ of $(D,P)$ and let $f^k$ be the Alexander numbering of $(D^k, P^k)$ obtained from $f$ by shifting by $k$.  Then $f^0, \dots, f^{m-1}$ induce a mod $m$ Alexander numbering of $\varphi_m^0(D,P)$.  Thus $\varphi_m^0(D,P)$  is mod $m$ almost classical. 

The method of construction of $\varphi_m(D,P)$ introduced in Section~\ref{sect:CyclicCover1} is a special case of the construction of $\varphi_m^0(D,P)$.  

\begin{thm}\label{thm:General}
For a virtual link diagram $D$ with a cut point $P$, a diagram $\varphi_m^0(D,P)$ is unique up to strong equivalence. 
\end{thm}

\begin{proof}
Let $D_1$ and $D_2$ be virtual link diagrams obtained from the same $(D, P)$ by the construction for $\varphi_m^0(D,P)$ introduced above.  
By definition of $\varphi_m^0(D,P)$, 
every classical crossing of $D_1$ (or $D_2$) can be labelled uniquely with $c^k$ for a classical crossing $c$ of $D$ and $k \in \{0, \dots, m-1\}$.  
Thus there is a natural bijection between the classical crossings of $D_1$ and those of $D_2$.  
By an ambient isotopy of ${\mathbb R}^2$, we may assume that $D_1$ and $D_2$ coincide in a regular neighborhood of every classical crossing.  Let $E$ denote the closure of the complement  of the regular neighborhoods of all classical crossings of $D_1$ (or of $D_2$) in  ${\mathbb R}^2$. 
The intersection $D_1 \cap E$ (or $D_2 \cap E$) consists of virtual paths which are 
properly immersed arcs or immersed loops in $E$.  

Let $A(D_1)$ (resp. $A(D_2)$) be the set of properly immersed arcs of $D_1 \cap E$ (resp. $D_2 \cap E$), and 
let $L(D_1)$ (resp. $L(D_2)$) be the set of immersed loops of $D_1 \cap E$ (resp. $D_2 \cap E)$.  

Let $a_1 \in A(D_1)$ and $a_2 \in A(D_2)$ be virtual paths starting with the same point 
$a_1(0)=a_2(0)$ 
in $\partial (D_1\cap E)= \partial (D_2\cap E)$, and let $a_1(1)$ and $a_2(1)$ be their terminal points in 
$\partial (D_1\cap E)= \partial (D_2\cap E)$.  We assert that $a_1(1) = a_2(1)$.  This is seen as follows. Let 
$E(D)$ be the complement of the regular neighborhoods of all classical crossings of $D$ in ${\mathbb R}^2$. 
The intersection $D \cap E(D)$ consists of virtual paths which are properly immersed arcs or immersed loops in $E(D)$. 
Let $a(0)$ be a point of $\partial (D \cap E(D))$ corresponding to $a_1(0)$ and let $a$ be the virtual path of $D \cap E(D)$ starting at $a(0)$.  Let $a(1)$ be the terminal point of $a$ in $\partial (D \cap E(D))$.  Then $a_1(0) = a_2(0) = a(0)^k$, 
$a_1(1)= a(1)^{k'}$ and $a_2(1) = a(1)^{k''}$ for some $k, k', k'' \in \{0, \dots, m-1\} = {\mathbb Z}_m$.  
Note that $k' -k$ is the sum of $-\epsilon(p)$ for all cut points $p \in P$ appearing on $a$, and so is 
 $k'' -k$.  Thus,  we see that $k' =k''$ and $a_1(1) = a_2(1)$.  Therefore, there is a bijection between 
 $A(D_1)$ and $A(D_2)$ such that corresponding arcs $a_1 \in A(D_1)$ and $a_2 \in A(D_2)$ have the same starting point and the same terminal point.  
 
Every loop of $L(D_1)$ (or $L(D_2)$) can be labelled as $\ell^k$ for a virtual path being an immersed loop $\ell$ of $D$ and $k \in \{0, \dots, m-1\}$.  Thus there is a bijection between $L(D_1)$ and $L(D_2)$.  

By detour moves,  replace virtual paths which are elements of $A(D_1)$ and $L(D_1)$ with 
the corresponding elements of $A(D_2)$ and $L(D_2)$, and we can obtain $D_2$ from  $D_1$.  This implies that  $D_1$ and $D_2$ are strongly equivalent. 
\end{proof}

We introduce another  method of construction of cyclic covering virtual link diagrams, which is a special case of  the method above. 
Let $(D, P)$ be a virtual link diagram with a cut system. 
Put $m$ copies of  $(D,P)$ in ${\mathbb R}^2$, say $(D^0, P^0), \dots, (D^{m-1}, P^{m-1})$,  
  such that all corresponding semi-arcs are in parallel as in Figure~\ref{fgAlexcov2p23} and all crossings between $D^k$ and $D^{k'}$  for $k \ne k'$ are virtual crossings.  Here semi-arcs of $D^{k+1}$ appears on the right of $D^{k}$ with respect to the orientation of $D$ as in Figure~\ref{fgAlexcov2p23}. See Figure~\ref{fgAlexcov1p2} (i) and (ii) for an example with $m=3$. 

\vspace{0.3cm}
\begin{figure}[h]
\centerline{
\includegraphics[width=4.cm]{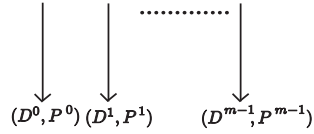}}
\caption{Parallel virtual link diagrams}\label{fgAlexcov2p23}
\end{figure}
\begin{figure}[h]
\centerline{
\includegraphics[width=12.cm]{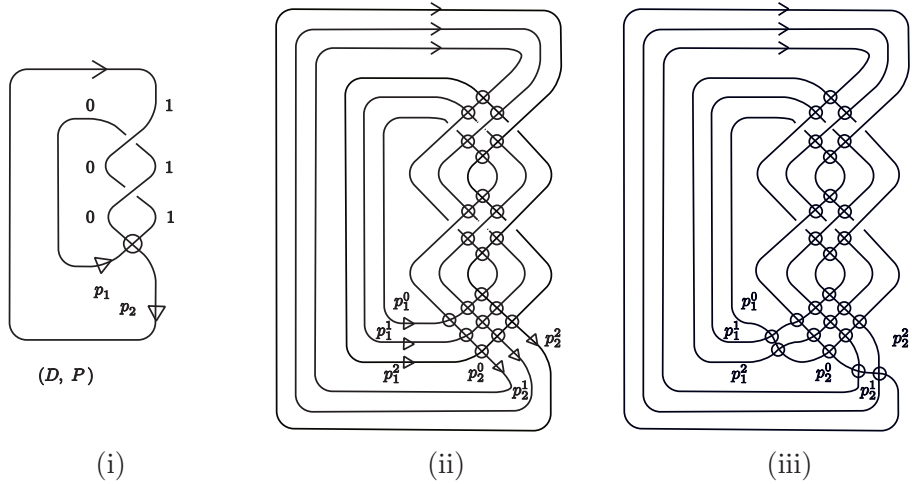}}
\centerline{(i)\hspace{4cm}(ii)\hspace{4cm}(iii)}
\caption{A $3$-fold cyclic covering virtual link diagram}\label{fgAlexcov1p2}
\end{figure}

For a cut point $p \in P$, let $p^k$ denote the corresponding cut point of $P^k$. 
Remove regular neighborhoods of all $p^k$  for $p \in P$ and $k \in \{0, \dots, m-1\} = {\mathbb Z}_m$ 
from $\cup_{k=0}^{m-1} D^k$, and connect the endpoints by virtual paths 
 as in Figure~\ref{fgAlexcov2p24} (iii) (resp. (iv)) if the cut point is coherent (resp. incoherent) as in 
 Figure~\ref{fgAlexcov2p24} (i) (resp. (ii)). 

\vspace{0.3cm}
\begin{figure}[h]
\centerline{
\includegraphics[width=12.cm]{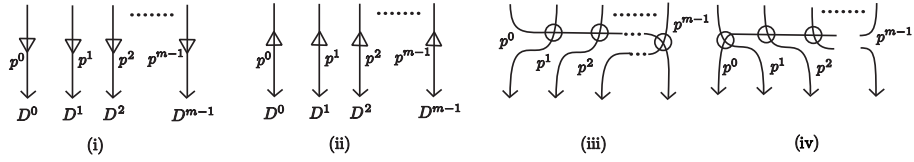}}
\caption{Replacement of neighborhoods of  cut points}\label{fgAlexcov2p24}
\end{figure}

Then we obtain a virtual link diagram.  Let us denote it by $\varphi_m^1(D,P)$.  
See Figure~\ref{fgAlexcov1p2} (iii) for an example. This concrete construction is also a special case of the general construction $\varphi^0_m(D,P)$.  
By Theorem~\ref{thm:General}, $\varphi_m(D,P)$, $\varphi_m^1(D,P)$ and $\varphi^0_m(D,P)$ are 
all strongly equivalent.  We call them {\it cyclic covering (virtual link) diagrams}.  

From this construction we see the following.
\begin{cor}\label{thmtwistknot}
Let $(D,P)$ be a virtual knot diagram with cut system. Then $\varphi_m(D,P)$ is an $m$-component virtual link diagram.
\end{cor}

\begin{proof}
Consider $\varphi_m^1(D,P)$.  
Since the  number of coherent cut points of $P$ equals that of incoherent cut points of $P$ 
(Corollary~\ref{cor:numbers}),  
the number of twists as in Figure~\ref{fgAlexcov2p24} (iii) appearing in $\varphi_m^1(D,P)$ 
equals that of the opposite twists as in Figure~\ref{fgAlexcov2p24} (iv).  Thus $\varphi_m^1(D,P)$ is  an $m$-component virtual link diagram, and so is $\varphi_m(D,P)$. 
\end{proof}

\section{Applications}\label{application}

First, we demonstrate how Theorem~\ref{thm2} is used to show that 
two virtual link diagrams are not equivalent. 

Let $(D,P)$ and $(D',P')$ be virtual link diagrams with cut points depicted in Figure~\ref{fgAlexapp1p} (i) and (ii). Then  $\varphi_3(D,P)$ and $\varphi_3(D',P')$ are as in the figure. 

\begin{figure}[h]
\centerline{
\includegraphics[width=10cm]{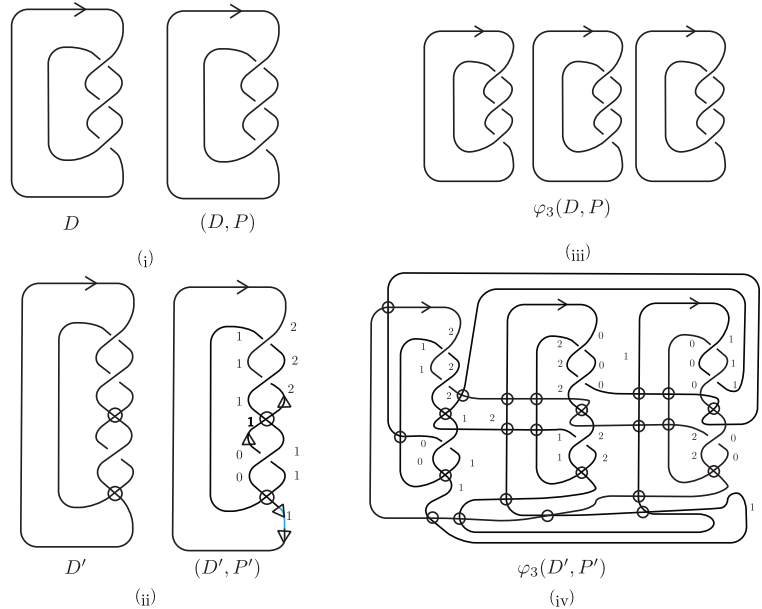}}
\caption{Example of mod 3 cyclic covering virtual link diagram}\label{fgAlexapp1p}
\end{figure}

It is easily seen that $\varphi_3(D,P)$ and $\varphi_3(D',P')$ are not equivalent, since any pair of components of $\varphi_3(D,P)$ have linking number $0$ and any pair of components of $\varphi_3(D',P')$ have linking number $1$.  By Theorem~\ref{thm2}, we conclude that $D$ and $D'$ are not equivalent.


Theorem~\ref{thm2} implies Theorem~\ref{thm:appli} below, which 
can be used to show that some virtual link diagrams are never equivalent to mod $m$ almost classical virtual link diagrams.  

\begin{lem} \label{lem:appli}
Let $D$ be a mod $m$ almost classical virtual link diagram. 
For any cut system $P$ of $D$,  $\varphi_m(D, P)$ is strongly equivalent to a virtual link diagram which is a disjoint union of $m$ copies of $D$. 
\end{lem}

\begin{proof}
There is a cut system $P_0$ of $D$ such that for each semi-arc of $D$, there are no cut points on it or there are $m$ coherent (or incoherent) cut points on it.  Each semi-arc of $D$ with $m$ coherent (or incoherent) cut points yields $m$ copies of such semi-arcs in the $m$ parallel copies of $D$, and $m$ virtual paths  in $\varphi_m(D, P_0)$ as in Figure~\ref{fgAlexcov3p}.  These $m$ virtual paths in $\varphi_m(D, P_0)$ can be 
 replaced with $m$ straight virtual paths by detour moves, and we obtain a disjoint union of $m$ copies of $D$.   This implies that  $\varphi_m(D, P_0)$ is strongly equivalent to the disjoint union of $m$ copies of $D$.  
 Thus $\varphi_m(D, P_0)$ is strongly equivalent to a disjoint union of $m$ copies of $D$.  
 By Lemma~\ref{lem2} (or Theorem~\ref{thm:General}), we see that $\varphi_m(D, P)$ is strongly equivalent to a  disjoint union of $m$ copies of $D$. 
\end{proof}

\begin{figure}[h]
\centerline{
\includegraphics[width=10.5cm]{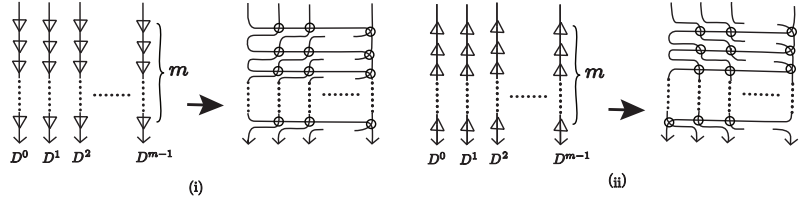}}
\caption{mod $m$ almost classical virtual link and its oriented cut points}\label{fgAlexcov3p}
\end{figure}
%


\begin{thm} \label{thm:appli}
If $\varphi_m(D,P)$ is not equivalent to a disjoint union of $m$ copies of $D$, then $D$ is never equivalent to a mod $m$ almost classical virtual link diagram.  
\end{thm}

\begin{proof}
Suppose that $D$ is equivalent to a mod $m$ almost classical virtual link diagram $D'$.  
By Lemma~\ref{lem:appli}, $\varphi_m(D', P')$ is equivalent to a disjoint union of $m$ copies of $D'$. 
By Theorem~\ref{thm2},  $\varphi_m(D,P)$ and $\varphi_m(D', P')$ are equivalent. Thus, 
$\varphi_m(D, P)$ is equivalent to a disjoint union of $m$ copies of $D'$, and hence equivalent to a disjoint union of $m$ copies of $D$.  This contradicts the hypothesis. 
\end{proof}

Let $D'$ be the virtual link diagram depicted in Figure~\ref{fgAlexapp1p}.  For the cut system $P'$ in the figure, $\varphi_3(D',P')$ is not equivalent to a disjoint union of $D$, since a pair of its components have linking number $1$.  By Theorem~\ref{thm:appli}, we can conclude that $D'$ is never equivalent to a mod $3$ almost classical virtual link diagram.

\vspace{5pt}
\noindent
{\bf Acknowledgement}\\
The author  would like to thank Seiichi Kamada and Shin Satoh for their fruitful conversation.

\end{document}